\documentclass[12pt,amscd]{amsart}
\usepackage{graphicx}
\usepackage{amsmath,amsxtra,amssymb,latexsym, amscd,amsthm,enumerate}
\usepackage{indentfirst}
\usepackage[mathscr]{eucal}
\usepackage{stackrel}
\usepackage{longtable}
\usepackage{float} 

\newtheorem{thm}{Theorem}[section]
\newtheorem{cor}[thm]{Corollary}
\newtheorem{lem}[thm]{Lemma}
\newtheorem{prop}[thm]{Proposition}
\theoremstyle{definition}

\DeclareMathOperator{\ext}{Ext}
\def\m {\mathfrak  m}

\begin{document}

\title[Buchsbaumness of the second powers of edge ideals]{Buchsbaumness of the  second powers of edge ideals}

\author{Do  Trong Hoang}

\email{dthoang@math.ac.vn}
\address{Institute of Mathematics, Vietnam Academy of Science and Technology, 18 Hoang Quoc Viet, 10307 Hanoi, Vietnam}

\author{Tran Nam Trung}

\email{tntrung@math.ac.vn}
\address{Institute of Mathematics, Vietnam Academy of Science and Technology, 18 Hoang Quoc Viet, 10307 Hanoi, Vietnam}


\date{}
\dedicatory{}
\commby{}

\dedicatory{Dedicated to Professor Le Tuan Hoa on the occasion of his sixtieth birthday}

\begin{abstract} We graph-theoretically characterize the class of graphs $G$ such that $I(G)^2$ are Buchsbaum.
\end{abstract}
\maketitle
\section*{Introduction}

Throughout this paper let $G=(V(G),E(G))$ be a finite simple graph without isolated vertices. An {\it independent set} in $G$ is a set of  vertices no two of which are adjacent to each other. The size of the largest independent set, denoted by $\alpha(G)$, is called the {\it independence number} of $G$. A graph is called {\it well-covered}
if every maximal independent set has the same size. A well-covered graph $G$ is a member of the class $W_2$ if the  remove any vertex of $G$ leaves a well-covered graph with the same independence number as $G$ (see e.g. \cite{P}).

Let $R=K[x_1,\ldots,x_n]$ be a polynomial ring of $n$ variables over a given field $K$. Let $G$ be a simple graph on the vertex set $V(G)=\{x_1, \ldots , x_n\}$. We associate to the graph $G$ a quadratic squarefree monomial ideal $$I(G) = (x_ix_j  \mid x_ix_j \in E(G)) \subseteq R,$$ which is called the {\it edge ideal} of $G$. We say that $G$ is {\it Cohen-Macaulay} (resp. {\it Gorenstein}) if $I(G)$ is a Cohen-Macaulay (resp. Gorenstein) ideal. It is known that $G$ is well-covered whenever it is Cohen-Macaulay (see e.g. \cite[Proposition $6.1.21$]{Vi}) and $G$ is in $W_2$ whenever it is Gorenstein (see e.g. \cite[Lemma 2.5]{HT}). It is a wide open problem to characterize graph-theoretically the Cohen-Macaulay (resp. Gorenstein) graphs. This problem was considered for certain classes of graphs (see \cite{HH, HHZ, HMT2, HT}).  Generally, we cannot read off the Cohen-Macaulay and  Gorenstein  properties of $G$  just from its structure because these properties   in fact depend on the characteristic of the base field $K$ (see \cite[Exercise $5.3.31$]{Vi} and \cite[Proposition 2.1]{HT}).

If we move on to the higher powers of $I(G)$, then we can graph-theoretically characterize $G$ such that $I(G)^m$ is Cohen-Macaulay (or Buchsbaum, or generalized Cohen-Macaulay)  for some $m\geqslant 3$ (and for all $m\geqslant 1$) (see \cite{GT,RTY1,TT}).  For the second power, we proved that $I(G)^2$ is Cohen-Macaulay if and only if $G$ is a {\it triangle-free} graph in $W_2$ (see \cite{HT}). As a consequence one can easily answer the question when $I(G)^2$ is generalized Cohen-Macaulay (see Theorem \ref{gCM}).

The remaining problem is to characterize $G$ such that $I(G)^2$ is Buchsbaum. For a vertex $v$ of  $G$, let $G_v$ be the induced subgraph $G\setminus (\{v\}\cup N_G(v))$ of $G$. In this paper, we will call $G$ a {\it locally triangle-free} graph if $G_v$ is triangle-free for any vertex $v$ of $G$. It is worth mentioning that \cite[Theorem $2.1$]{RTY2} and \cite[Theorem $4.4$]{HT} suggest that $G$ may be a locally triangle-free Gorenstein graph if $I(G)^2$ is Buchsbaum. So it is natural to characterize such graphs. Note that they are in $W_2$ by \cite[Proposition $3.7$]{HT}. In this paper we will settle this problem when we obtain a characterization of locally triangle-free graphs in $W_2$ (see Theorem $\ref{main-theorem}$). Let $C_n^c$ be the complement of the cycle $C_n$ of length $n$. Then,

\medskip

\noindent {\bf Theorem 1 {\rm (Theorem \ref{thmGorenstein})}. } {\it Let $G$ be a locally triangle-free graph. Then  $G$ is  Gorenstein if and only if $G$ is either a triangle-free graph in $W_2$, or $G$ is isomorphic to one of  $C_n^c$ ($n\geqslant 6$), $Q_{9}$, $Q_{12}$, $P_{10}$ or $P_{12}$ (see Figure \ref{fig_fourgraphs}).}
\begin{figure}[H]
\begin{tabular}{ccc}
 \includegraphics[scale=0.55]{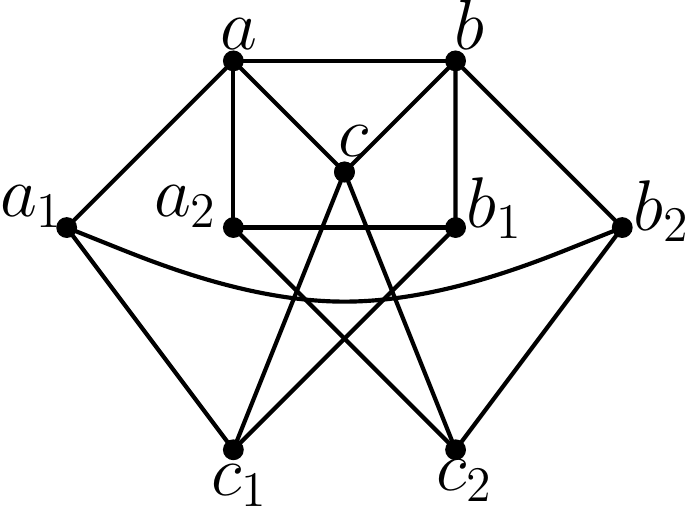}&\qquad\qquad&
 \includegraphics[scale=0.55]{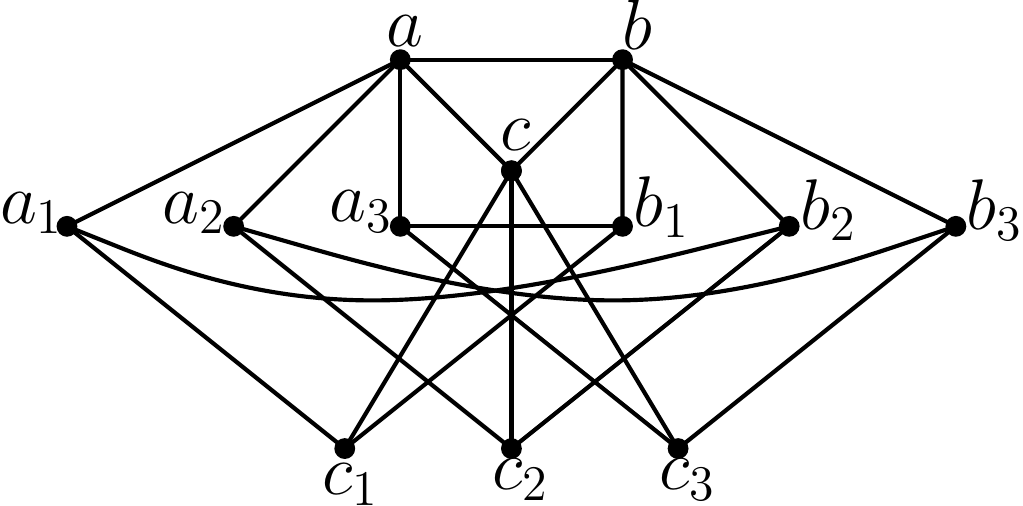}
 \\[6pt] {$Q_{9}$} && {$Q_{12}$}\\[6pt]
 \includegraphics[scale=0.45]{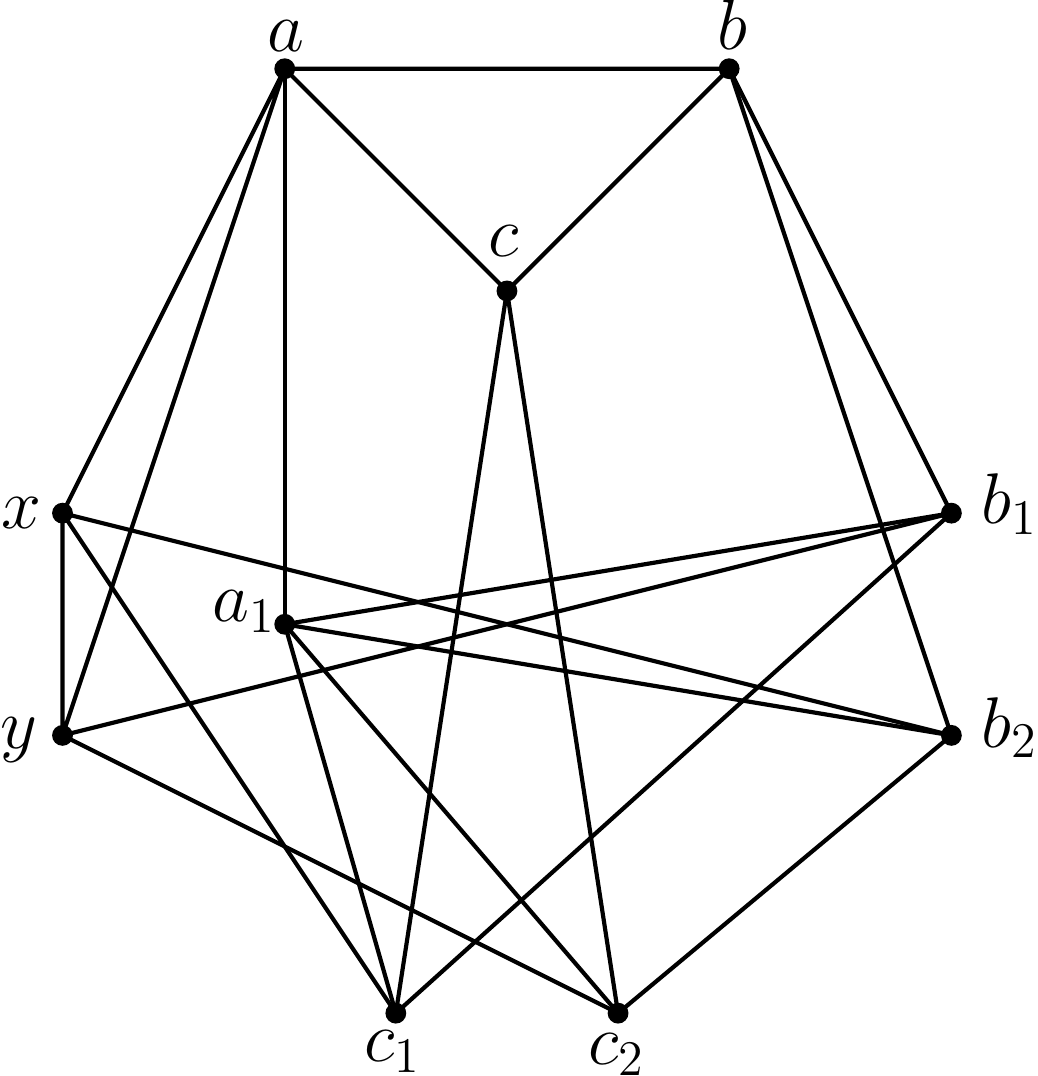}&&
 \includegraphics[scale=0.45]{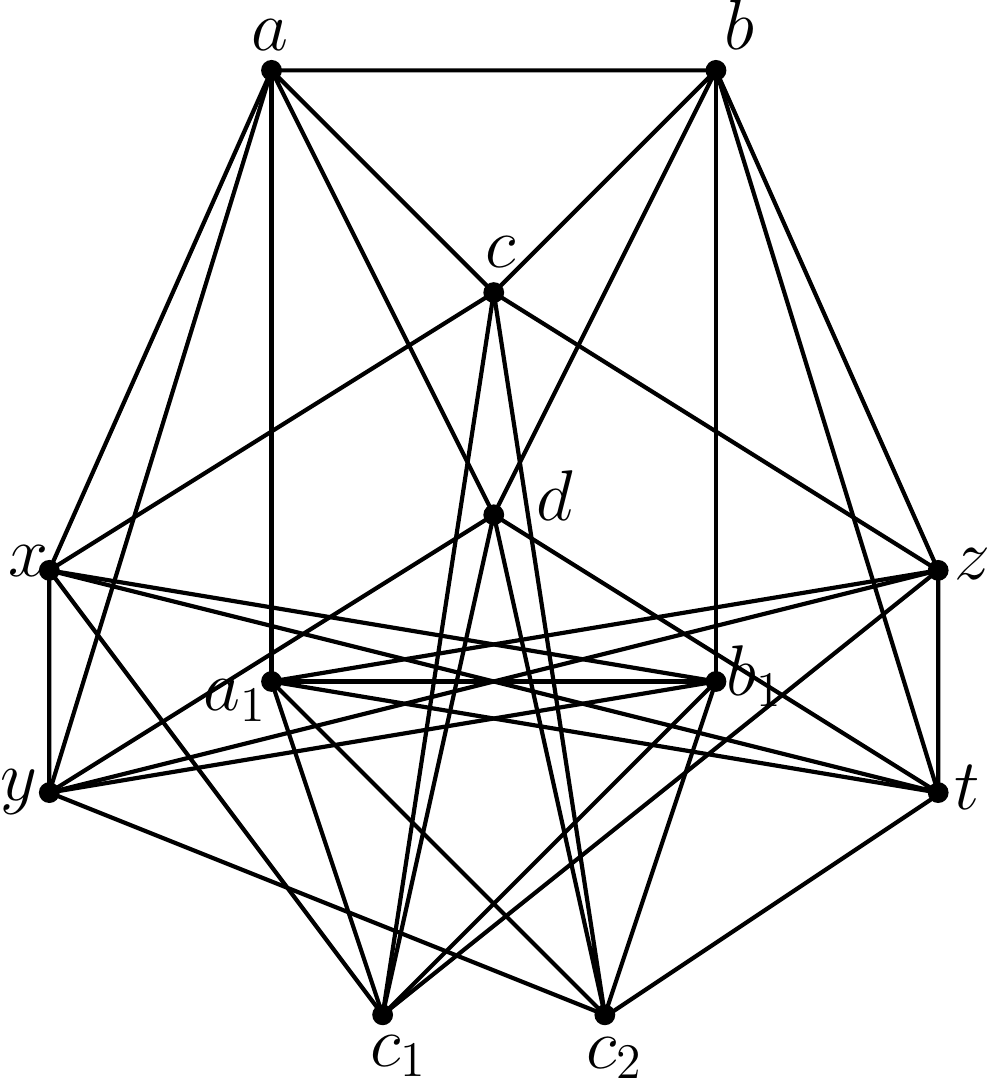}
 \\[6pt]
{$P_{10}$}&&
{$P_{12}$} \\[6pt]
\end{tabular}
\caption{For graphs $Q_9$, $Q_{12}$, $P_{10}$ and $P_{12}$.}\label{fig_fourgraphs}
\end{figure}

\medskip

Now let $B_n$ ($n\geqslant 4$) be the  graph with the edge set $\{x_ix_j| 3\leqslant i+1< j\leqslant n\}$. Using this theorem we can characterize graphs $G$ such that $I(G)^2$ are Buchsbaum.

\medskip

\noindent {\bf Theorem 2 {\rm (Theorem \ref{thmBuchsbaum})}. } {\it  Let   $G$ be a graph. Then  $I(G)^2$ is  Buchsbaum if and only if  $G$ is either a triangle-free graph in $W_2$, or isomorphic to one of $K_n$  ($n\geqslant 3$), $C_n^c$ ($n\geqslant 6$), $B_n$  $(n\geqslant 4)$,  $Q_9, Q_{12}, P_{10}$ or $P_{12}$.
}

\medskip

The paper is organized as follows. In Section 1 we recall some basic notations, and terminologies from Graph theory. In Section $2$ we investigate the local structure of locally triangle-free graphs in $W_2$.  Section $3$ is devoted  to classifying the class of  
locally triangle-free graphs in $W_2$. In the last section we graph-theoretically characterize graphs $G$ for which $I(G)^2$ are Buchbaum.

\section{Preliminaries}

Let $R := K[x_1,\ldots, x_n]$ be the polynomial ring over a field $K$ and $\m := (x_1,\ldots,x_n)R$  the maximal homogeneous ideal of $R$. Let $H^i_{\m}(R/I)$ denote the $i$-th local cohomology module of $R/I$ with respect to $\m$. A residue class ring $R/I$ is called a {\it generalized Cohen-Macaulay} (resp. {\it  Buchsbaum}) ring if $H^i_{\m}(R/I)$   has finite length (resp.
the canonical map $$\ext_R^i(R/\m,S/I) \to H_{\m}^{i}(R/I)$$ is surjective) for all $i<\dim(R/I)$ (see \cite{CST, SV}).

\medskip

First we address the problem of characterizing graphs $G$ such that  $I(G)^2$  are generalized Cohen-Macaulay.

\begin{thm}\label{gCM} Let $G$ be a simple graph. Then, $I(G)^2$ is generalized Cohen-Macaulay if and only if:
\begin{enumerate}
\item $G$ is well-covered;
\item Every nontrivial component of $G_v$ is a triangle-free graph in $W_2$ for any vertex $v$ of $G$.
\end{enumerate}
\end{thm}

\begin{proof}
Follow from \cite[Corollaries $2.3$ and $3.10$]{HMT} and \cite[Theorem $4.4$]{HT}.
\end{proof}

Next we recall some terminologies  from Graph theory. Let $G$ be a simple graph on the vertex set $V(G)$ and the edge set $E(G)$. An edge $e\in E(G)$ connecting two vertices $u$ and $v$ will also be  written as $uv$ (or $vu$). In this case, we say that $u$ and $v$ are adjacent. For a subset $S$ of $V(G)$, the neighborhood of  $S$ in $G$ is the set  $$N_G(S) := \{v\in V (G)\setminus S\mid uv\in E(G) \text{ for some }u\in S\},$$
and the close neighborhood of $S$ in $G$ is $N_G[S] := S \cup N_G(S)$.  We denote  by $G[S]$  the induced subgraph of $G$ on the vertex set $S$, denote $G\setminus S$ by $G[V\setminus S]$,   and denote $G_S$ by $G\setminus N_G[S]$.  For an edge  $ab$ of $G$, we write $G_{ab}$ stands for $G_{\{a,b\}}$. The number $\deg_G(v):=|N_G(v)|$ is called the {\it degree} of  $v$ in  $G$.

\begin{lem}\label{locally-well-covered} {\rm (\cite[Lemma $1$]{FHN})} If $G$ is a well-covered graph and $S$ is an independent set of $G$, then $G_S$ is well-covered. Moreover, $\alpha(G_S) = \alpha(G) - |S|$.
\end{lem}

\begin{lem} \label{locally-W2} {\rm (\cite[Lemma 7]{Ho})} Let $G$ be  a graph in $W_2$ and $S$ an independent set of $G$. If $|S| < \alpha(G)$, then $G_S$ is in $W_2$. In particular, $G_S$ has no isolated vertices.
\end{lem}

A graph $G$ is called {\it bipartite} if its vertex set can be partitioned into subsets $A$ and $B$ so that every edge has one end in $A$ and one end in $B$;  such a partition is called a {\it bipartition} of the graph $G$ and denoted by $(A, B)$. It is well known that  $G$ is bipartite if and only if $G$ has no odd cycles (see e.g. \cite[Theorem 4.7]{BM}).  

\begin{lem} {\rm (\cite[Lemma 12]{Ho})} \label{bipartite}
 If $G$ is a bipartite graph in $W_2$, then $G$ consists of disjoint edges.
\end{lem}

An $(s-1)$-path of $G$ is a  sequence  of its edges $u_1u_2,u_2u_3,\ldots,u_{s-1}u_s$ and will be denoted by $u_1\ldots u_s$.  An  $s$-cycle ($s\geqslant 3$) is a path $u_1\ldots u_su_1$, where $u_1,\ldots,u_s$ are distinct vertices; it will be denoted by $(u_1\ldots u_s)$. A  3-cycle is called a {\it triangle}.  A graph $G$ is called triangle-free if it has no triangles; and $G$ is a locally triangle-free graph if  $G_v$ is triangle-free for every vertex $v$.

\begin{lem}{\rm (\cite[Lemma 10]{Ho})}\label{locally-edge}   Let $G$ be a locally triangle-free graph in $W_2$ and let $ab$ be an edge of $G$. Then, $G_{ab}$ is either empty or well-covered with  $\alpha(G_{ab})  = \alpha(G)-1$.
\end{lem}

Let $G_1=(V_1,E_1)$ and $G_2=(V_2,E_2)$ be two disjoint graphs, i.e. $V_1 \cap V_2=\emptyset$. Then, the join of $G_1$ and $G_2$, denoted by $G_1*G_2$, is the graph with the vertex set $V_1 \cup V_2$ and the  edge set $E_1\cup E_2\cup\{uv\mid u\in V_1, v\in V_2\}$. Note that $G_1$ and $G_2$ are two induced subgraphs of $G_1*G_2$. 

A {\it simplicial complex} $\Delta$ on the vertex set $V$ is a collection of subsets of $V$ closed under taking subsets; that is, if $\sigma \in \Delta$ and $\tau\subseteq \sigma$ then $\tau\in \Delta$. For a graph $G$, let $\Delta(G)$ be the set of independent sets of $G$. Then $\Delta(G)$ is a simplicial complex which is  the so-called {\it independence complex} of $G$. 

The condition that $G$ is a join of its two  proper subgraphs  can be represented via the {\it connectivity} of $\Delta(G)$.

\begin{lem} \label{disconnected} A graph $G$ is a join of   its two proper subgraphs if and only if $\Delta(G)$ is disconnected.
\end{lem}
\begin{proof} Assume that $G = G_1 * G_2$, where $G_1$ and $G_2$ are two non-empty graphs. Then, $\Delta(G) = \Delta(G_1)\cup \Delta(G_2).$
Since $\Delta(G_1) \cap \Delta(G_2)=\{\emptyset\}$, we have $\Delta(G)$ is disconnected.

Conversely, if $\Delta(G)$ is disconnected, then it can write as a union of two simplicial complexes
$$\Delta(G) = \Delta_1 \cup \Delta_2$$
such that $\Delta_1\cap \Delta_2=\{\emptyset\}$ and $V_i\ne\emptyset$, where $V_i$ is the set of vertices of $\Delta_i$, for $i=1,2$.

Let $G_i := G[V_i]$ for $i=1,2$. We will show that $G = G_1 * G_2$, or equivalently $E(G) = E(G_1 * G_2)$. Indeed, it is obvious that $E(G) \subseteq E(G_1 * G_2)$. We now prove the reverse inclusion. Let $v_1v_2$ be an edge of $G_1*G_2$. If $v_1v_2$ is an edge of either $G_1$ or $G_2$, then $v_1v_2\in E(G)$. Hence, we may assume that $v_i\in V(G_i)$ for $i=1,2$.

Since each $v_i$ is a vertex of $\Delta_i$, we have $\{v_1,v_2\}\notin \Delta(G_1) \cup \Delta(G_2) =\Delta(G)$. In other words, $v_1v_2\in E(G)$, so that $E(G) = E(G_1 * G_2)$, and the lemma follows.
\end{proof}

Locally triangle-free graphs $G$ in $W_2$ with $\alpha(G)\leqslant 2$ have simple structure. Namely,

\begin{prop}\label{small-cases}   Let  $G$ be a locally triangle-free graph  in $W_2$ with $n$ vertices. Then,
\begin{enumerate}
\item If $\alpha(G)=1$, then $G$ is $K_n$ with $n\geqslant 2$;
\item If $\alpha(G) = 2$, then $G$ is $C_n^c$ with $n\geqslant 4$.
\end{enumerate}
\end{prop}
\begin{proof} If $\alpha(G)=1$, then   $G = K_n$ is a complete graph. Since $G$ is in $W_2$, $n\geqslant 2$.

If $\alpha(G)=2$, for each $v\in V(G)$, $G_v$ is a triangle-free graph in $W_2$ by Lemma \ref{locally-W2} and  $\alpha(G_v)=1$ by Lemma $\ref{locally-well-covered}$. Thus,  $G_v$ is just an edge. It follows that $\deg_G(v)=n-3$ for any $v\in V(G)$. It yields $\deg_{G^c}(v)=2$ for all $v\in V(G^c)$, so $G^c$ is an $n$-cycle. Since $\alpha(G)=2$, we get $n\geqslant 4$, as required.
\end{proof}

\section{The structure of Neighborhoods}

In this section we explore the local structure of  locally triangle-free graphs $G$ in $W_2$ with $\alpha(G)\geqslant 3$. Namely, let $(abc)$ be a triangle in $G$, and let $A:=N_G(a)\setminus N_G[b]$. Then,
\begin{enumerate}
\item Either $A\in \Delta(G)$ or $G[A]$ is one edge and $\alpha(G)-2$ isolated vertices.
\item If $G$ is not a join of its two  proper subgraphs, then $G_{ab}$ is not empty.
\end{enumerate}
The first result is the following.

\begin{lem}\label{GA} Let $G$ be a locally triangle-free graph in $W_2$ with $\alpha(G)\geqslant 3$. Let $(abc)$ be a triangle in $G$ such that $G_{ab}\ne \emptyset$. Let $A:=N_G(a)\setminus N_G[b]$ (see Figure \ref{structureG}). If $A$ is not an  independent set of $G$, then
\begin{enumerate}
\item $G[A]$ consists of one edge and $\alpha(G)-2$ isolated vertices, and
\item $G_{ab}$ has $2(\alpha(G)-2)$ vertices.
\end{enumerate}
\end{lem}
\begin{proof} In order to prove the lemma we divide into the following claims:

\begin{figure}[H]
\includegraphics[scale=0.6]{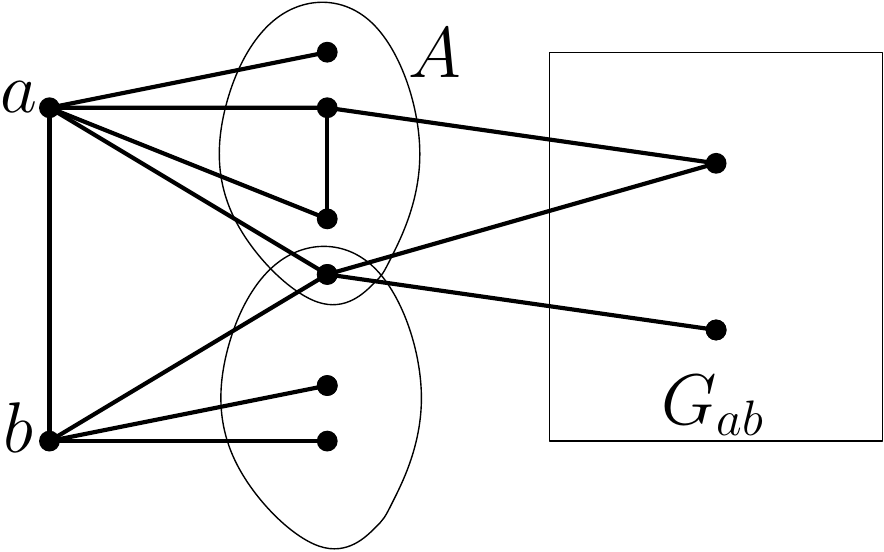}\\
\medskip
\caption{The structure of graph $G$.}
\label{structureG}
\end{figure}


{\it Claim  $1$:  If $xy\in E(G[A])$ and   $v\in V(G_{ab})$, then $v$ is adjacent in $G$ to  exactly one of the two vertices   $x$ and  $y$.}

\medskip

Indeed, if $vx, vy \notin E(G)$, then $G_v$ contains a triangle $(axy)$. If $v$  is adjacent  to both $x$ and $y$, then  $(vxy)$ is the triangle in $G_b$. In both cases $G$ is not locally triangle-free, a contradiction.

\medskip

{\it Claim  $2$:  Assume that  $xy\in E(G[A])$,   $uv\in E(G_{ab})$ and $ux\in E(G)$. Then, $vy\in E(G)$ and $uy,vx\notin E(G)$.}

\medskip

Indeed, by Claim $1$, $ux\in E(G)$ implies  $uy\notin E(G)$.  If $vx\in E(G)$, then $G_b$ has the triangle $(xuv)$, a contradiction. Hence $vx\notin E(G)$, and $vy\in E(G)$ by Claim $1$.

\medskip

{\it Claim  $3$:  $G_{ab}$ is bipartite.}

\medskip

Indeed, assume that $G_{ab}$  has an odd cycle of length $2k+1$, say $(z_1\ldots z_{2k+1})$, for some $k\geqslant 1$. Let $xy\in E(G[A])$.  Since $z_1z_2\in E(G_{ab})$, by Claims $1$ and $2$, we may assume that $z_1x\in E(G)$ and so $z_2y\in E(G)$. Since  $z_2z_3\in E(G_{ab})$, by Claim $2$ we have $z_3x\in E(G)$. Repeating  this argument for  $z_3z_4,\ldots,z_{2k}z_{2k+1},z_{2k+1}z_1$, we finally obtain $yz_1\in E(G)$. Then, $G_b$ has a triangle $(z_1xy)$, a contradiction. Therefore, $G_{ab}$ must be bipartite.

\medskip

{\it Claim  $4$:  $G[A]$ is bipartite.}

\medskip

Indeed, assume  that $G[A]$  has an odd cycle, say $(z_1\ldots z_{2m+1})$, of length $2m+1$ for some $m\geqslant 1$. Let $v\in V(G_{ab})$. By Claim $2$ we may assume that $vz_1\in E(G)$ so that $vz_2\notin E(G)$. Repeating  this argument for   $z_2z_3,\ldots,z_{2m}z_{2m+1},z_{2m+1}z_1$, we finally obtain $vz_1\notin E(G)$,  a contradiction. Hence, $G[A]$ is bipartite, as claimed.

\medskip

Now let $S$  be the set of isolated vertices of $G[A]$ and let  $\Gamma_1,\ldots, \Gamma_t$ be the  connected  components of $G[A\setminus S]$. Note that $t\geqslant 1$ because $A$ is not an independent set of $G$. By Claim $4$, each $\Gamma_i$ is bipartite and let $(A_i, B_i)$ be its bipartition.

\medskip

{\it Claim  $5$:  $S\ne\emptyset$ and every vertex of $G_{ab}$ is adjacent to just one vertex in $S$.}

\medskip

Indeed, since $G_{ab} \ne \emptyset$, let $v\in V(G_{ab})$ and  let $H$ be a connected component of $G_{ab}$, which contains $v$. If $H$ is just one point $v$,  then set  $C=\{v\}$, and $D=\emptyset$. Otherwise, by Claim 3, $H$ is a bipartite and we let $(C, D)$ be a bipartition of $H$, where $v\in C$. Let $x$ be an arbitrary element of $C$. By repeating Claim 1, we can see that for each $\Gamma_i$, $x$ is adjacent to all vertices in $B_i$ but not adjacent to any vertex in $A_i$. If $D\ne \emptyset$, then there is $y\in D$ such that $xy\in E(G)$. By Claim 2, $y$ is adjacent to all vertices in $A_i$ but not adjacent to any vertex in $B_i$. Applying the same argument to all edges between $C$ and $D$, we  conclude that all vertices in $C$ (resp. $D$) have the above properties as $x$ (resp. $y$). Let $X:=\cup_{i=1}^t A_i$. Then $X\cup \{b\}\in \Delta(G)$ and  we can see that $G_{X \cup \{b\}}$ is a bipartite graph with a bipartition $(S, V(G_{ab}\backslash N_G(X))$. By Lemma $\ref{bipartite}$, this graph is a disjoint union of edges. Note that $v\in C\subseteq V(G_{ab}\backslash N_G(X))$, thus
$$|S| = |V(G_{ab}\backslash N_G(X))| \geqslant |C| \geqslant 1,$$
and thus $v$ is is adjacent to just one vertex in $S$, as claimed.

\medskip

By Claim 5, $G_{S\cup \{b\}} = \Gamma_1\cup \cdots\cup \Gamma_t$, so $G_{S\cup \{b\}}$ is a bipartite graph. Thus $\Gamma_i$ is an edge for $i=1,\ldots,t$, by Lemma \ref{bipartite}.  Let   $S = \{p_1,\ldots, p_m\}$ for $m\geqslant 1$; and let $\Gamma_i$ be the edge $a_ib_i$ for $i=1,\ldots,t$. Observe that $\alpha(G_{S\cup \{b\}})= \alpha(G[A\setminus S])=t$, and hence $\alpha(G)=t+m+1$ by Lemma \ref{locally-well-covered}.

Since each vertex of $G_{ab}$ is adjacent to just one vertex in $S$, the set $V(G_{ab})$ can be partitioned into $V(G_{ab}) = V_1\cup \cdots\cup V_m$, where $V_i$ is the set of all vertices of $G_{ab}$ which  are adjacent   to $v_i$.  Moreover, $V_i\ne\emptyset$ because every vertex in $S$ is adjacent to some vertex of $G_{ab}$. Now we show that $V_i\in\Delta(G)$ for all $i$. Indeed, if $G[V_i]$ has an edge, say $xy$, for some $i$, then $(xyv_i)$ would be a triangle in $G_b$, which is impossible as $G_b$ is triangle-free, and then $V_i\in\Delta(G)$.

 \medskip

 {\it Claim  $6$:  $|V_i| = t+1$ for $i=1,\ldots,m$.}

\medskip

Indeed, let $v\in V_i$. we may assume that $va_1,\ldots,va_s \in E(G)$. Let $U := S \setminus\{p_i\} \cup\{b, v\}$. Then, $\alpha(G_U) = \alpha(G) - |U| = (t+m+1) - (m+1) = t$. On the other hand, $G_U$ is a bipartite graph with bipartition $(\{b_1,\ldots, b_t\}, V_i\setminus\{v\})$. By Lemma $\ref{bipartite}$, this graph is just disjoint edges, so $|V_i| = t+1$, as claimed.

 \medskip

In summary, we have proved that $G[A]$ consists of $t$ disjoint edges and $m$ isolated vertices; $|V(G_{ab})| = m(t+1)$ and $\alpha(G) = t+m+1$. Hence, it remains to prove $t=1$.

Assume on the contrary that $t\geqslant 2$. Write $V_1=\{u_1,\ldots,u_{t+1}\}$. By Claim $2$ we may assume that $u_{t+1}a_1,\ldots, u_{t+1}a_t\in E(G)$ and $u_{t+1}b_1,\ldots, u_{t+1}b_t\notin E(G)$. We also can assume that $u_ib_i\in E(G)$ for all $i=1,\ldots,t$; and $u_ib_j\notin E(G)$ for all $1\leqslant i \neq j\leqslant t$. These facts together with Claim $2$ help us conclude  $u_ia_j\in E(G)$ for all $i\ne j$. Now $v_1$ is an isolated vertex of $G_{(S\setminus v_1)\cup \{b,a_1,a_2\}}$,  but this fact contradicts Lemma \ref{locally-W2}. Hence, $t=1$, and the proof of the lemma is complete.
\end{proof}

\begin{lem}\label{empty-set} Let $G$ be a triangle-free graph in $W_2$. Assume that $V(G)$ can be partitioned into $V(G) = S \cup T \cup U$, such that
\begin{enumerate}
\item $|S|+|T|\geqslant 2$;
\item Every vertex in $U$ is adjacent to all vertices in $S\cup T$.
\end{enumerate}
Then, $U = \emptyset$.
\end{lem}
\begin{proof} Assume on the contrary that $U\ne\emptyset$. As $|S| + |T| \geqslant 2$, we may assume that $S\ne\emptyset$. Since  $G$ is a triangle-free graph, $U\in\Delta(G)$. Let $u\in U$. Then, $G_u = G[U\setminus\{u\}]$, so $G_u$  must be empty by Lemma $\ref{locally-W2}$. It follows that $\alpha(G)=1$, so $G$ is just an edge, and  $|V(G)|=2$. On the other hand, $|V(G)| = |S|+|T|+|U| \geqslant 3$, a contradiction.
\end{proof}

\begin{lem}\label{Intersection} Let $G$ be a locally triangle-free graph in $W_2$ with $\alpha(G)\geqslant 3$ such that $G$ is not a join of two proper subgraphs. Assume that $v_1v_2$ is an edge of $G$ such that $G_{v_1v_2} = \emptyset$. Then, $N_{G}(v_1)\cap N_{G}(v_2)=\emptyset$.
\end{lem}

\begin{proof} Since $\alpha(G_{v_1}) = \alpha(G)-1\geqslant 2$  and $G_{v_1}$ is in $W_2$ by Lemma $\ref{locally-W2}$, there is an edge $ab$ in $G_{v_1}$. Let
$$ A:= N_G(a)\setminus N_G[b], B:= N_G(b)\setminus N_G[a], \text{ and } C:=N_G(a)\cap N_G(b).$$
Note that  $v_2\in C$ and $v_1\in V(G_{ab})$. Figure \ref{fig_Gab} depicts   this situation.

\begin{figure}[H]
\includegraphics[scale=0.6]{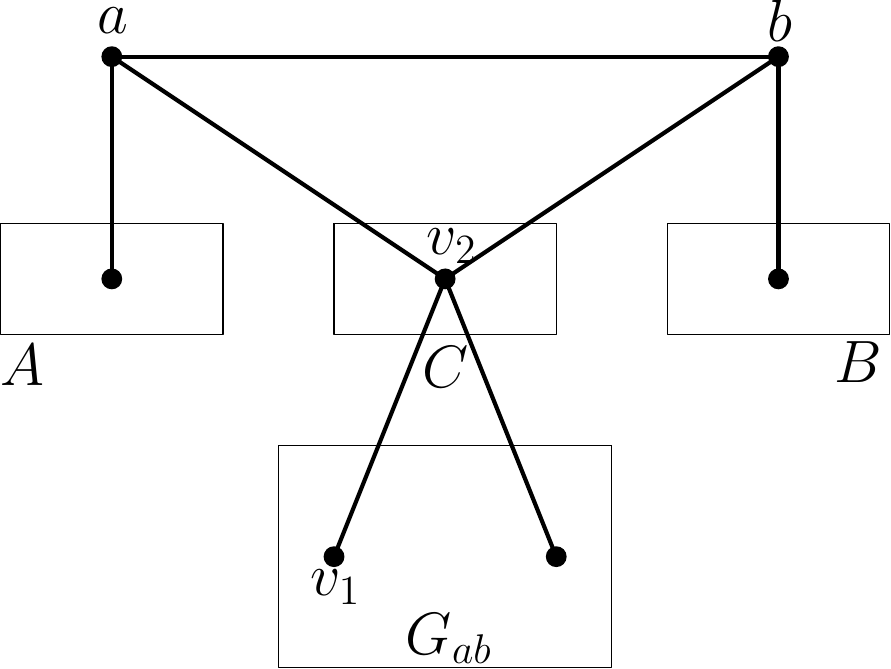}\\
\caption{A configuration for  graph $G$.}
\label{fig_Gab}
\end{figure}


{\it  Claim 1: Every vertex in $C$ is adjacent to all vertices in $V(G_{ab})$.}

\medskip

Indeed, assume on the  contrary that   $cv\notin E(G)$ for some $c\in C$ and $v\in V(G_{ab})$, then $(abc)$ would be a triangle in $G_v$,  a contradiction, as claimed.

\medskip

{\it  Claim 2: $A$ and $B$ are not empty sets.}

\medskip

Indeed, if $A=B=\emptyset$, by Claim $1$ we obtain $G =  G[C] * G[\{a,b\}\cup V(G_{ab})]$, a contradiction. Hence, we are able to assume that $A\ne\emptyset$.

Assume that $B=\emptyset$. Then, every vertex of $G_{ab}$ is adjacent to all vertices in $A$. Because if $uv\notin E(G)$ for some $u\in V(G_{ab})$ and $v\in A$, by Claim $1$ we get $b$ as an isolated vertex of $G_{\{u,v\}}$.  It,  however,  contradicts Lemma $\ref{locally-W2}$. Now if $G_{ab}$ has an edge, say $xy$, then $(xyv)$ would be a triangle in $G_b$ for any $v\in A$, a contradiction. Hence, $G_{ab}$ is a totally disconnected graph.  On the other hand, $G_{ab} = G_a$, so that $G_a$ is also a totally disconnected graph. But then it contradicts Lemma $\ref{locally-W2}$, thus $B\ne \emptyset$, and the claim follows.

\medskip

Let $S$ be the set of isolated vertices of $G[A]$ and let $T$ be the set of isolated vertices of $G[B]$. By Lemma $\ref{GA}$, we see that $G[A]$ (resp. $G[B]$) is either totally disconnected or one edge and $\alpha(G)-2$ isolated vertices, so  $S$ (resp. $T$) is not empty.

\medskip

{\it  Claim 3: If a vertex in $C$ is adjacent to a vertex in $S$ (resp. $T$),  it must be adjacent to all vertices in $A$ (resp. $B$).}

\medskip

Indeed, assume that $cv\in E(G)$ for some $c\in C$ and $v\in S$. If $G_{bc}\ne\emptyset$, then $G_{bc}$ is well-covered and  $\alpha(G_{bc})  = \alpha(G)-1$ by Lemma $\ref{locally-edge}$. Since  $cv\in E(G)$,  $V(G_{bc}) \subseteq A\setminus \{v\}$ and $G_{bc}$ is well-covered, we have
$$\alpha(G_{bc}) \le \alpha(G_{bc}\setminus S)+|S\setminus \{v\}|  \le \alpha(G[A]\setminus S)+|S| -1 = \alpha(G[A]) -1.$$
It follows that $\alpha(G)=\alpha(G_{bc})+1\leqslant \alpha(G[A])$. On the other hand,
$$\alpha(G)\geqslant \alpha(G[A\cup\{b\}]) = \alpha(G[A])+1,$$
a contradiction. Thus, $G_{bc} =\emptyset$, and thus  $c$ is adjacent to every vertex in $A$, as claimed.

\medskip

We now let
\begin{eqnarray*}
 C_1 &:=& \{c\in C \mid c \text{ is adjacent to all vertices in } A\},\\
 C_2&:=&\{c\in C \mid c \text{ is not adjacent to any vertex in }S\}.
 \end{eqnarray*}
By Claim $2$, the set $C$ has a partition $C=C_1\cup C_2$.  We next prove that $C_1 =\emptyset$.

\medskip

{\it Claim 4:  If $C_1\ne \emptyset$,  then every vertex in $C_1$ is adjacent to all vertices in $B$.}

\medskip

Indeed, assume on the contrary that $cv\notin E(G)$ for some $c\in C_1$ and $v\in B$.  If  $B$ is an independent set, i.e. $B=T$, then $c$ is not adjacent to any vertex in $B$ by Claim $3$. Combining with Lemma \ref{locally-edge}, we get  $G_{ac}=G[B]$, and so $|B|=\alpha(G_{ac})=\alpha(G)-1\ge 2$. Thus, $B\cup\{c\}$ is a maximal independent set of $G$, and thus  $C\setminus N_G[c] \subseteq N_G(B)$. By Claim $3$, every vertex in $B$ is adjacent to all vertices in $C\setminus N_G[c]$. Let $z$ be a vertex in $B\setminus \{v\}$. Since $V(G_{\{c,z\}}) = B\backslash \{z\}$,  $v$ is an isolated vertex of $G_{\{c,z\}}$, which contradicts Lemma $\ref{locally-W2}$.

If $B\notin \Delta(G)$, then  by Lemma \ref{GA} we have $G[B]$ is just one edge, say $xy$, and isolated vertices, say $q_1,\ldots,q_m$, where $m=\alpha(G)-2$. Hence, $T=\{q_1,\ldots,q_m\}$ and $m\geqslant 1$.  Note that $cq_i\notin E(G)$ for any $i=1,\ldots, m$ by Claim $3$.

If $cx,cy\in E(G)$,  then $G_{q_1}$ has a triangle $(cxy)$, a contradiction.

If $cx,cy\notin E(G)$, then $V(G_{c})= (C \setminus N_G[c])\cup B$. Since $G_{c}$ has no isolated vertices by Lemma $\ref{locally-W2}$, one has $q_1u\in E(G)$ for some $u\in C\setminus N_G[c]$.  Note that $u$ is adjacent to all vertices in $B$ by Claim $3$. But then $G_{c}$ has a triangle  $(uxy)$, a contradiction.

If  $c$ is adjacent to either $x$ or $y$ but not both, we may assume that  $cx\in E(G)$ and $cy\notin E(G)$. Then, $V(G_{c}) = (C \setminus N_G[c]) \cup   \{y,q_1,\ldots, q_m\}$. Hence, every edge of $G_c$ has an endpoint in $C \setminus N_G[c]$. Together with Claim $3$ it follows that any vertex of $G_c$ that is adjacent to $q_1$   is adjacent to $y$ as well, so $N_{G_c}(q_1)\subseteq N_{G_c}(y)$. But then, $q_1$ is an isolated vertex of $G_{\{c,y\}}$, which contradicts Lemma $\ref{locally-W2}$, and then claim follows.

\medskip

{\it Claim 5:  $C_2 \ne \emptyset$.}

\medskip

Indeed,  assume on the contrary that $C_2=\emptyset$ so that $C_1=C\ne \emptyset$. By Claims $1$ and $4$,  every vertex in $C_1$ is adjacent to all vertices in $\{a,b\}\cup A\cup B\cup V(G_{ab})$.  It follows that  $$G =   G[C_1]*G[\{a,b\}\cup A\cup B\cup V(G_{ab})],$$
a contradiction, so  $C_2\ne \emptyset$, as claimed.

\medskip

{\it Claim 6: If $A\in \Delta(G)$, then $|A| = \alpha(G)-1$.}

\medskip

Indeed, since $C_2\ne \emptyset$ by Claim $5$, we can take $c\in C_2$. Then, $c$ is not adjacent to any vertex in $A$ by Claim $3$, so  that $G_{bc} = G[A]$. Now by applying Lemma $\ref{locally-edge}$ we obtain
$$|A| = \alpha(G[A]) = \alpha(G_{bc}) = \alpha(G)-1,$$
as claimed.

\medskip

{\it Claim 7:  $C_1 = \emptyset$ and $C_2\in\Delta(G)$.}

\medskip

By Lemma $\ref{GA}$ and Claim $6$ we get $|S| \geqslant \alpha(G)-2$, so $\alpha(G_S) = \alpha(G)-|S| \leqslant 2$. Since $G_S$ is a triangle-free graph in $W_2$, it must be an edge, or two disjoint edges or a pentagon by Proposition \ref{small-cases}. Consequently, $\deg_{G_S}(x) \leqslant 2$ for every vertex $x$ of $G_S$. Since $|C_2| \leqslant \deg_{G_S}(b)$, we obtain $|C_2| \leqslant 2$. Together with Claim $5$, this fact yields $|C_2| =1$ or $|C_2| = 2$.

If $|C_2| = 1$, then $C_2 =\{c\}$ for some vertex $c$. We can partition $V(G_{c})$ into  $$V(G_{c}) = (A\setminus N_G(c))\cup (B\setminus N_G(c)) \cup (C_1 \setminus N_G(c)).$$
Since $|A\backslash N_G(c)|\cup |B\backslash N_G(c)|\ge |S|+|T|\ge 2$, together with Lemma $\ref{empty-set}$ and Claim 4, this fact gives $C_1 \setminus N_G(c) =\emptyset$. In other words, $c$ is adjacent to all vertices in $C_1$. Thus, by Claims $1$ and $3$  we conclude that  all vertices in $C_1$ are adjacent to all  vertices in $\{a,b,c\}\cup A\cup B\cup V(G_{ab})$. If $C_1\ne \emptyset$, then
$$G =   G[C_1]*G[ \{a,b,c\}\cup  A \cup B\cup V(G_{ab})],$$
a contradiction. Hence,  $C_1=\emptyset$.

If $|C_2| = 2$, in this case we have $G_S$ is a pentagon and $|S| = \alpha(G)-2$. By Claim $6$, $A$ is not the independent set in $G$. Thus, by Lemma \ref{GA}, $G[A]$ is a disjoint union of one  edge, say $xy$, and isolated vertices in $S$. Let $C_2 =\{c,c'\}$ for some $c,c'\in V(G)$. Since $\{b, c,c', x,y\} \subseteq V(G_S)$, we may assume that $G_S$ is the pentagon with the edge set $\{bc,cx,xy,yc',c'b\}$. In particular, $C_2\in \Delta(G)$. Since we can partition $V(G_{C_2})$ into
$$V(G_{C_2}) =  S \cup (B\setminus N_G(C_2)) \cup (C_1 \setminus N_G(C_2)).$$
By Lemma $\ref{empty-set}$ and Claim 4, we have $C_1 \setminus N_G(C_2)=\emptyset$. It follows that $c'$ is adjacent to all vertices in $C_1\setminus N_G(c)$, and $c$ is adjacent to all vertices in $C_1\setminus N_G(c')$. Together with Lemma $\ref{empty-set}$ and  the following partition of $V(G_c)$,
$$V(G_{c}) = (S\cup \{y\}\cup \{c'\}) \cup (B\setminus N_G(c)) \cup (C_1 \setminus N_G(c)),$$
this fact yields $C_1\setminus N_G(c) =\emptyset$. Similarly, $C_1\setminus N_G(c')=\emptyset$, i.e. every vertex in $C_2$ is adjacent to all vertices in $C_1$. Thus, if $C_1\ne\emptyset$, then we have
$$G =   G[C_1]*G[ \{a,b,c,c'\}\cup (A \cup B\cup V(G_{ab})],$$
a contradiction, and the claim follows.

\medskip

We now return to prove the lemma. Since $C = C_2$ by Claim $7$, we have $v_2\in C_2$. Consequently, $v_2$ is not adjacent to any vertex in $S\cup T$. Since $G_{v_1v_2}=\emptyset$, $v_1$ is adjacent to all vertices in $S\cup T$.

Now assume on the  contrary that $N_G(v_1) \cap N_G(v_2) \ne  \emptyset$. Let $v_3\in N_G(v_1) \cap N_G(v_2)$. Note  that $C_2\in \Delta(G)$ and
$$N_G(v_2) \subseteq \{a,b\} \cup V(G_{ab})\cup (A\setminus S) \cup (B\setminus T),$$
so either $v_3\in V(G_{ab})$ or $v_3\in (A\setminus S) \cup (B\setminus T)$.

Assume that $v_3$ is a vertex of $G_{ab}$. Then, $v_3$ is not adjacent to any vertex in $S\cup T$. Because    assume on the contrary that $v_3p\in E(G)$ for some $p\in S$ (similarly,  $p\in T$). Then, $(pv_1v_3)$ would be a triangle in $G_b$, a contradiction. It follows that $S\cup\{b,v_3\}$ is an independent set in $G$, and so
$|S|\le \alpha(G)-2$. Moreover, by Claim $6$ and Lemma $\ref{GA}$,  $|S|\geqslant \alpha(G)-2$, and so $|S|=\alpha(G)-2$; and $G[A\setminus S]$ is just an edge, say $xy$. Since $S\cup\{b,v_3\}$ is an independent set of $G$, we imply that $v_3$ is adjacent to both $x$ and $y$, and so $(xyv_3)$ is a triangle in $G_b$, a contradiction.

Assume that $v_3\in (A\setminus S) \cup (B\setminus T)$. We may assume that $v_3\in A\setminus S$. In this case $A$ is not an independent set in $G$, so $G[A]$ consists of  one edge, say $xy$, and isolated vertices in $S$   with   $|S|=\alpha(G)-2$. Then either $v_3 = x$ or $v_3 = y$. If $v_2$ is adjacent to both $x$ and $y$, then $G_{bv_2} = G[S]\ne\emptyset$. Therefore, $\alpha(G) =\alpha(G_{bv_2})+1=|S|+1$, a contradiction. We now may assume that $v_2x\notin E(G)$, so that $v_3=y$. On the other hand, since $G_{v_1v_2} =\emptyset$, $v_1x\in E(G)$. Hence, $(v_1xy)$ is a triangle in $G_b$, a contradiction.

Therefore, we must have $N_{G}(v_1)\cap N_{G}(v_2)= \emptyset$, and the proof of the lemma is complete.
\end{proof}

From Lemma \ref{Intersection} we obtain the following result:

\begin{cor}\label{triangle-edge} Let $G$ be a locally triangle-free graph in $W_2$ with $\alpha(G)\geqslant 3$ such that $G$ is not a join of two proper subgraphs. Then, for any edge $ab$ lying in a triangle in $G$ we have $\alpha(G_{ab}) =\alpha(G)-1$. In particular, $G_{ab}\ne \emptyset$.
\end{cor}
\begin{proof} Assume $ab$ is in the triangle $(abc)$ for some vertex $c$ of $G$, so that $c\in N_G(a)\cap N_G(b)$. In particular,
$N_G(a)\cap N_G(b) \ne\emptyset$, so $G_{ab}\ne\emptyset$ by Lemma $\ref{Intersection}$. The lemma now follows from Lemma $\ref{locally-edge}$.
\end{proof}

\section{Locally Triangle-free graphs in $W_2$}

In this section we characterize locally triangle-free graphs in $W_2$. First we deal with such graphs that are not triangle-free. Thus, we assume that $G$ is a locally triangle-free graph in $W_2$ it satisfies:
\begin{enumerate}
\item $\alpha(G)\geqslant 3$;
\item $G$ is not a join of its two proper subgraphs;
\item $G$ has a triangle $(abc)$.
\end{enumerate}
Let  $$A := N_G(a)\setminus N_G[b], B := N_G(b)\setminus N_G[a], \text{ and } I := N_G(a) \cap N_G(b).$$
Then, by Corollary $\ref{triangle-edge}$, we have $G_{ab}, G_{bc}$ and $G_{ca}$ are not empty, and
$$\alpha(G_{ab}) = \alpha(G_{bc}) =\alpha(G_{ca}) =\alpha(G)-1.$$

We will classify $G$ via the structure of $G[A], G[B]$ and $G_{ab}$.

\begin{lem}\label{singleton}
\begin{enumerate}
\item Every vertex in $I$ is adjacent to all vertices of $G_{ab}$.
\item Every isolated vertex of $G[A]$ is not adjacent to any vertex in $I$.
\item If $A$ is an independent set of $G$ then $I=\{c\}$.
\end{enumerate}
\end{lem}
\begin{proof} $(1)$ If there are $x\in I$ and $y\in V(G_{ab})$ such that $x$ is not adjacent to $y$. Then, $G_y$ has a triangle $(abx)$, a contradiction.

$(2)$ If there are an isolated vertex of $G[A]$, say $v$, and a vertex in $I$, say $u$, such that $uv\in E(G)$, then by statement $(1)$ we would have $G_{bu}$ is an induced subgraph of $G[A\setminus\{v\}]$. In this case, $\alpha(G_{bu}) \leqslant \alpha(G[A\setminus\{v\}]) =\alpha(G[A])-1$. Since $(abu)$ is a triangle in $G$, by Corollary \ref{triangle-edge}, we get $\alpha(G_{bu}) =\alpha(G)-1$. Thus, $\alpha(G) =\alpha(G_{bu})+1 \leqslant \alpha(G[A])$. On the other hand, $\alpha(G[A]) =\alpha(G[A\cup\{b\}]) -1 \leqslant  \alpha(G)-1$, a contradiction.

$(3)$ By Statements $(1)$ and $(2)$ we get $G_{bc} = G[A]$. In particular, $|A| = \alpha(G)-1$. It follows that $\alpha(G_A) = 1$, so it is an edge. By Statement $(2)$ we have $G[I\cup\{b\}]$ is an induced subgraph of $G_A$, so $G_A$ is just the edge $bc$ and so $I = \{c\}$.
\end{proof}

\begin{lem}\label{NB} If $A\notin \Delta(G)$, then:
\begin{enumerate}
\item $\alpha(G) =3$;
\item $G[A]$ is just one edge and one isolated vertex;
\item $G_{ab}$ is just two isolated vertices.
\end{enumerate}
\end{lem}
\begin{proof} Let $m = \alpha(G)-2$. By Lemma $\ref{GA}$ we have $G[A]$ consists of one edge and $m$ isolated vertices, and $G_{ab}$ has $2m$ vertices.

If $G_{ab}$ is a totally disconnected graph, then $\alpha(G_{ab}) = 2m$ so that $\alpha(G) = 2m+1$. Together with $\alpha(G)=m+2$, this equality yields $m=1$; and the lemma follows.

Assume that $G_{ab}$ is not a totally disconnected graph. We will prove that this assumption leads to a contradiction. Let $C := N_G(c) \setminus N_G[a]$.  Observe that $c$ is adjacent to every vertex in $G_{ab}$ by Lemma $\ref{singleton}$. Thus, $G_{ab}$ is an induced subgraph of $G[C]$, and thus $C$ is not an independent set of $G$. By Lemma $\ref{GA}$ where we replace $a$ and $b$ by $c$ and $a$ respectively, $G[C]$ consists of one edge and $m$ isolated vertices. Consequently, $G_{ab}$ is just one edge and $2m-2$ isolated vertices. It follows that $\alpha(G_{ab}) = 2m-1$. From $\alpha(G) = \alpha(G_{ab})+1$ and $\alpha(G)=m+2$, we obtain $m =2$. Hence, $\alpha(G)=4$ and $G[C] = G_{ab}$.

Recall that $C\notin \Delta(G)$. If $G_{ac}$ is a totally-disconnected graph, then $\alpha(G)=3$ as above. Therefore, $G_{ac}$ is not a totally disconnected graph. By Lemma $\ref{singleton}$ we deduce that $G_{ac}$ is also an induced subgraph of $G[B]$.

Since $V(G_{ac}) \subseteq B$ by Lemma $\ref{singleton}$, we have $G_{ac}$ is an induced subgraph of $G[B]$, so $B\notin \Delta(G)$. By the argument above we get $G[B] = G_{ac}$ consists of one edge and $2$ isolated vertices. By symmetry, we also have $G[A] = G_{bc}$.

Assume that $G[A]$ is the edge $xy$ and two isolated vertices $a_1$ and $a_2$; $G[B]$ is the edge $zt$ and two isolated vertices $b_1$ and $b_2$; and $G[C]$ is one edge $uv$ and two isolated vertices $c_1$ and $c_2$.

Note that $G_b=G[A\cup C]$. We now explore the structure of this graph.  Firstly, we have all vertices in $C$ are adjacent to exactly one of two vertices $x$ and $y$. Indeed, if $wx,wy\in E(G)$ for some $w\in C$, then $G_b$ has a triangle $(wxy)$, a contradiction. If $wx,wy\notin E(G)$, then $(axy)$ is a triangle in  $G_w$, a contradiction. Thus,  we may assume that $xv, yu, c_1x\in E(G)$ and $xu,yv,c_1y\notin E(G)$.  Since $\{b,v,c_1\}\in \Delta(G)$ and $\alpha(G_{\{b,v,c_1\}})=1$, $G_{\{b,v,c_1\}}$ is just an edge; and this edge must be $yc_2$, and so  $c_2x\notin E(G)$.  Similarly,  since  $c_1\in V(G_{\{b,v,c_2\}}) \subsetneq \{c_1,a_1,a_2\}$, we assume $a_1c_1\in E(G)$  and $a_1c_2\notin E(G)$.  Thus, $V(G_{\{b,u,c_1\}})\subseteq \{c_2,a_2\}$, and so $c_2a_2\in E(G)$, $a_2c_1\notin E(G)$. Furthermore, since $a_2,c_2\in V(G_{\{b,x,a_1\}})$, we have $ua_1\in E(G)$ and $va_1\notin E(G)$. Next, since $v,a_2\in V(G_{\{b,y,a_1\}})$, $va_2\in E(G)$ and $ua_2\notin E(G)$. It follows that
$$E(G_b) =\{xy, uv, xv, xc_1, yu,yc_2, a_1u,a_1c_1, a_2v, a_2c_2\}.$$

In the same way we  may assume  $E(G_a) =\{zt, uv, zv, zc_1, tu,tc_2, b_1u,b_1c_1, b_2v, b_2c_2\}.$
By symmetry,  $G_c = G[A\cup B]$ has the same structure as  $G_a$ and $G_b$. It follows that $z$ is adjacent to either $a_1$ or $a_2$. If $z$ is adjacent to $a_1$, then $G_{c_2}$ has the triangle $(za_1c_1)$. If $z$ is adjacent to $a_2$, then $G_{a_1}$ has the triangle $(za_2v)$. Thus, $G$ is not locally triangle-free in both cases, a contradiction, and the lemma follows.
\end{proof}

\begin{lem}\label{Qs} If $A, B, V(G_{ab}) \in \Delta(G)$, then $G$ is isomorphic to either $Q_9$ or $Q_{12}$.
\end{lem}
\begin{proof}  By Lemma $\ref{singleton}$ we have $G_{bc} = G[A]$,   $ G_{ac}= G[B] $, and $N_G(a)\cap N_G(b) = \{c\}$.  In particular, $|A| = |B|=\alpha(G)-1$.

Let $C:=V(G_{ab})$ and $s:=\alpha(G)-1$. Then  $N_G(c)=\{a,b\}\cup C$, and $|A| = |B| = |C| =s\ge 2$. Assume that $A =\{a_1,\ldots,a_s\}$; $B = \{b_1,\ldots,b_s\}$ and $C =\{c_1,\ldots,c_s\}$.

Note that $G_b$ is a bipartite graph with bipartition $(A,C)$. So $G_b$ is disjoint edges by Lemma $\ref{bipartite}$. Hence, we may assume that $E(G_b) =\{a_1c_1,\ldots,a_sc_s\}$. Similarly, we may assume that $E(G_a) = \{b_1c_1, \ldots, b_sc_s\}$. Together   with this Lemma \ref{bipartite} again,  $E(G_c) = \{a_{\sigma(1)}b_1,\ldots, a_{\sigma(s)}b_s\}$, where   $\sigma$  is a permutation of the set $\{1,\ldots,s\}$.

Therefore, $G[A\cup B\cup C]$ consists of disjoint cycles, say $C_1,\ldots,C_t$. Moreover  the length of each $C_i$ is a multiple of $3$, say $3s_i$,  for $s_i\geqslant 1$.  If $s_i=1$ for some $i$, then $C_i$ is   the form $(a_jb_jc_j)$ for some $j$, and so  $C_i$ is a triangle of $G_{c_m}$ for any $m\ne j$, a contradiction. Thus, $s_i\geqslant 2$. This yields
$\alpha(C_i) =\lfloor 3s_i/2 \rfloor = s_i+\lfloor s_i/2\rfloor\geqslant s_i+1$. Thus,
$$\alpha(G[A\cup B\cup C]) = \sum_{i=1}^t \lfloor 3s_i/2 \rfloor \geqslant \sum_{i=1}^t (s_i+1) = s+t.$$
Combining with $\alpha(G[A\cup B\cup C])\leqslant \alpha(G)=s+1$, we obtain $t=1$. This means that $\alpha(G[A\cup B\cup C])$ is a cycle of length $3s$. From the equality $\alpha(G[X\cup Y\cup Z]) \leqslant \alpha(G)$ we have $\lfloor 3s/2 \rfloor \leqslant s+1$, or equivalently $\lfloor s/2 \rfloor \leqslant 1$. This forces either $s=2$ or $s=3$.

If $s=2$,  $G$ is isomorphic to $Q_9$. Otherwise, $s=3$ and $G$ is isomorphic to $Q_{12}$.
\end{proof}

\begin{lem}\label{P10}  If  $A\notin \Delta(G)$ and $B\in\Delta(G)$, then $G$ is isomorphic to $P_{10}$.
\end{lem}
\begin{proof} By Lemma $\ref{NB}$, we have $\alpha(G)=3$, $G[A]$   consists of one edge and one isolated vertex, and $G_{ab}$ is two isolated vertices. By Lemma $\ref{singleton}$, we imply that $G[B] = G_{ac}$, and $N_G(a)\cap N_G(b) = \{c\}$. So $|B| =\alpha(G)-1=2$, i.e. $G[B]$ is just two isolated vertices.

Let $G[A]$ be the edge $xy$ and one isolated vertex $a_1$; $G[B]$ be two isolated vertices $b_1$ and $b_2$; $G_{ab}$ be two isolated vertices $c_1,c_2$.   This yields $G$ has a vertex set $V(G) = \{a,b,c, x,y, a_1, b_1,b_2, c_1,c_2\}$ (see Figure \ref{figurelem34}).

\begin{figure}[H]
\includegraphics[scale=0.5]{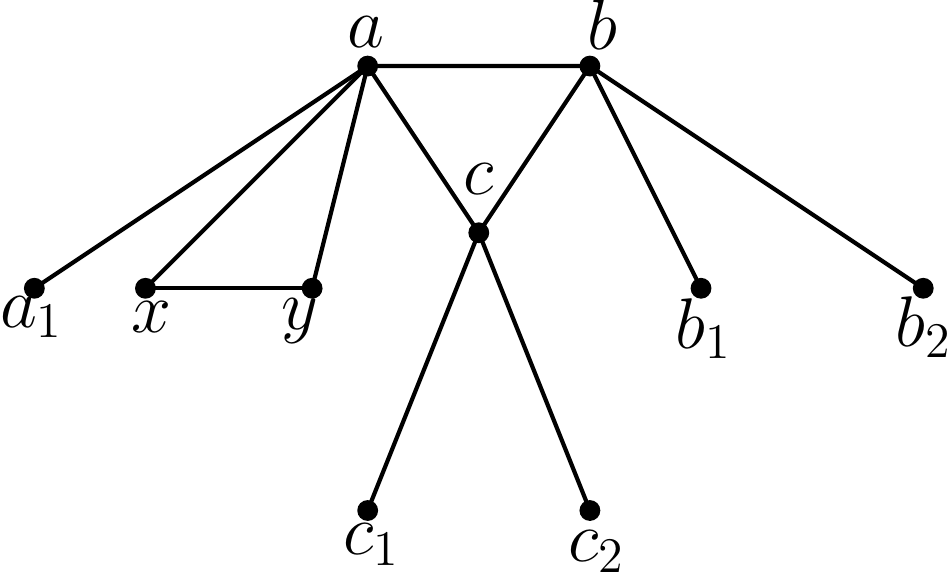}\\
\caption{The configuration for graph $G$.}\label{figurelem34}
\end{figure}

As $G_a$ is a bipartite graph with bipartition $(\{b_1,b_2\}, \{c_1,c_2\})$, it is just two disjoint edges by Lemma $\ref{bipartite}$, and so we may assume that $E(G_a) = \{b_1c_1, b_2c_2\}$.

Since $V(G_b) =\{x,y,a_1,c_1,c_2\}$ and $G_b$ is a triangle-free graph in $W_2$,  it must be a pentagon. Hence, we may assume $E(G_b) = \{xy, xc_1,c_1a_1, a_1c_2,c_2y\}$.

Note that $V(G_{c_1})=\{y,a,b,b_2,c_2\}$, so we have $E(G_{c_1}) = \{ab, bb_2, b_2c_2,c_2y, ya\}$. Hence, $b_2y\notin E(G)$. Since $(axy)$ is not a triangle in $G_{b_2}$, we must have $x\notin V(G_{b_2})$, or equivalently $b_2x\in E(G)$. Similarly, $b_1x\notin E(G)$ and $b_1y\in E(G)$.

Observe that $V(G_c) \subseteq \{x,y,a_1, b_1,b_2\}$. Since $\alpha(G_c) = 2$, we conclude that $G_c$ is either two disjoint edges or a pentagon.

Assume that $G_c$ is just two disjoint edges. Then, by Lemma $\ref{singleton}$ we have $a_1, b_1, b_2\in V(G_c)$, so the remaining vertex is either $x$ or $y$. By symmetry, we may assume it is $x$, i.e. $cx\in E(G)$.  It follows that $N_G(x) =\{a,c,y,b_2,c_1\}$, so $G_x$ must be two disjoint edges $bb_1$ and $a_1c_2$. Thus, $a_1b_1\notin E(G)$, and thus $N_G(b_1)=\{b,y,c_1\}$, and $|V(G_{b_1})| = 6$. On the other hand, $G_{b_1}$ must be either two disjoint edges or a pentagon, so $|V(G_{b_1})| \leqslant 5$, a contradiction.

Therefore  $G_c$ is a pentagon, and thus $V(G_c) = \{x,y,a_1, b_1,b_2\}$. Recall that $yb_2,xb_1\notin E(G)$, so $E(G_c) = \{xy,xb_2, b_2a_1,a_1b_1, b_1y\}$. It follows that
\begin{align*} E(G) = \{&ab, ac, aa_1, ax, ay, bc, bb_1, bb_2, cc_1, cc_2, xy, xb_2, xc_1, yb_1, yc_2, a_1b_1, a_1b_2, \\
&a_1c_1, a_1c_2, b_1c_1, b_2c_2\},
\end{align*}
so $G$ is isomorphic to $P_{10}$.
\end{proof}

\begin{lem} \label{P12}  If $A,B\notin \Delta(G)$, then $G$ is isomorphic to $P_{12}.$
\end{lem}
\begin{proof} Let $C := V(G_{ab})$. By Lemma $\ref{NB}$ we have $\alpha(G)=3$, both $A$ and $B$    consist  of one edge and one isolated vertex, and $G[C]$ is two isolated vertices. We may assume that $G[A]$ is one  edge $xy$ and an isolated vertex $a_1$; $G[B]$ is one edge $zt$ and an isolated vertex $b_1$; and $G[C]$ is two isolated vertices $c_1$ and $c_2$.

\medskip

{\it Claim  $1$: $I\in\Delta(G)$.}

\medskip

Indeed, assume on the  contrary that $uv\in E(G)$ for some $u,v\in I$. By Lemma \ref{singleton},  $a_1$ is not adjacent to any vertex in $I$. Thus, $(buv)$ is a triangle in $G_{a_1}$, a contradiction. Hence, $I\in \Delta(G)$, as claimed.

\medskip

Since $G_a$ is a triangle-free graph in $W_2$ and $V(G_a) =\{z,t,b_1, c_1,c_2\}$, $G_a$ must be a pentagon. Because $zt\in V(G)$, we may assume that $E(G_a) = \{zt, tc_2, c_2b_1, b_1c_1, c_1 z\}$.  Similarly, by symmetry we may assume that $E(G_b) = \{xy, yc_2, c_2a_1,a_1c_1,c_1x\}$ (see Figure \ref{fig_lem35}).

\begin{figure}[H]
\includegraphics[scale=0.5]{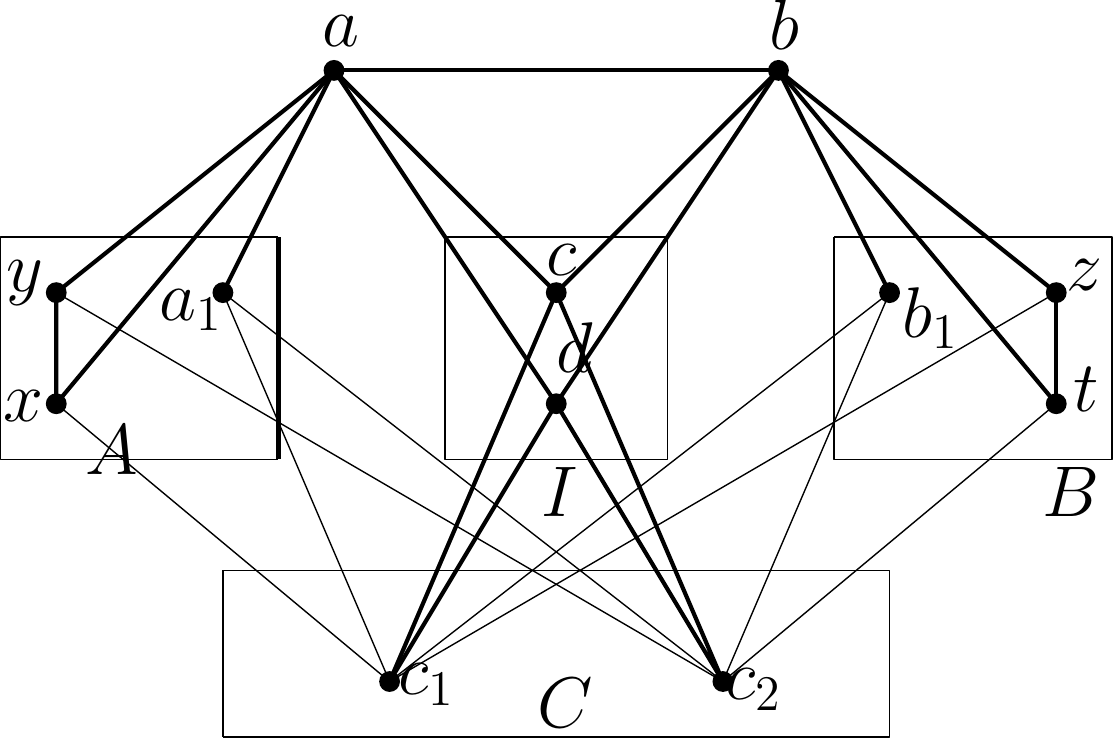}\\
\caption{The configuration for graph $G$.} \label{fig_lem35}
\end{figure}

\medskip

{\it Claim  2: $yt,xz \notin E(G)$ and $yz, xt\in E(G)$.}

\medskip

Indeed, as $V(G_{c_1}) =\{a,b, t,c_2,y\}$, we have $G_{c_1}$ is a pentagon, and so $E(G_{c_1}) =\{ab,bt,tc_2,c_2y,ya\}.$  It follows that $yt\notin E(G)$. Together with $G_{y}$   a triangle-free graph,  we imply that  $yz\in E(G)$. Similarly, $xz\notin E(G)$ and $xt\in V(G)$.

\medskip

{\it Claim  3: $a_1b_1 \in E(G)$.}

\medskip

Indeed, if $a_1b_1\notin E(G)$, then $I\cup\{a_1,b_1\}$ is an independent set of $G$. Since $\alpha(G)=3$, we have $|I| = 1$, so that $I =\{c\}$. Observe that $\{a, a_1, c,z,t\} \subseteq V(G_{b_1})$. Together with the fact that $G_{b_1}$ is a triangle-free graph in $W_2$ with $\alpha(G_{b_1}) =2$, it implies that $G_{b_1}$ is a pentagon with $V(G_{b_1}) = \{a, a_1, c,z,t\}$. Hence, $b_1x,b_1y\in E(G)$. But then $G_{a_1}$ has the triangle $(b_1xy)$, a contradiction, and the claim follows.

\medskip

{\it Claim  4: $a_1z, a_1t, b_1x, b_1y \in E(G)$.}

\medskip

Indeed, we only prove $a_1z\in E(G)$ and others are proved similarly. Assume on the contrary that $a_1z\notin V(G)$. Then, $G_z$ has a triangle $(a_1b_1c_2)$, a contradiction. Therefore, $a_1z\in E(G)$, as claimed.

\medskip

{\it Claim 5: $|I| =2$.}

\medskip

Indeed, since $\alpha(G_{\{x,z\}})=1$, $G_{\{x,z\}}$   must be an edge. Since $c_2\in V(G_{\{x,z\}})$, we imply that $G_{\{x,z\}}$ is the edge $c_2d$ for some $d\in I$. Note that $G_d$ is a triangle-free graph in $W_2$ and $\alpha(G_d)=2$, so it is either a pentagon or two disjoint edges. Since $G_d$ contains  3-path  $xb_1a_1z$,  $G_d$ must be a pentagon. On the other hand, $E(G_d)\subseteq \{a_1,b_1,z,x\}\cup (I\setminus\{d\})$. It follows that $|I\setminus\{d\}| \ne 0$, so $|I|\geqslant 2$. By Claim $1$ and Lemma $\ref{singleton}$, we get $I\cup\{a_1\}$ is an independent set of $G$. Since $\alpha(G)=2$, $|I| \leqslant 2$. It yields $|I| = 2$, as claimed.

\medskip

So now we may assume that $I = \{d,c\}$. In particular, $V(G_d) = \{c, a_1,b_1,z,x\}$, and hence, $cx,cz \in E(G)$. Note also that $dx,dz\notin E(G)$.

\medskip

{\it Claim  $6$: $cy, ct \notin E(G)$.}

\medskip

Indeed, if $cy\in E(G)$, then $G_{a_1}$ has a triangle $(xyc)$, a contradiction, and then $cy\notin E(G)$. Similarly, $ct\notin E(G)$.

\medskip

Now using Claim $6$, from the graph $G_c$ we get $dy,dt\in E(G)$.   In summary, we obtain:
\begin{align*}
E(G) = \{&ab, aa_1, ac,ad, ax, ay,
bb_1, bc, bd, bt, bz,
cx, cz, cc_1,cc_2,
dy, dt, dc_1, dc_2, a_1b_1, \\
&a_1z,a_1t, a_1c_1,a_1c_2, b_1x, b_1y,b_1c_1,b_1c_2, xy,xt, xc_1,yz,yc_2, zt,zc_1,tc_2
\},
\end{align*}
so $G$ is isomorphic to $P_{12}$.
\end{proof}

We are in position to prove the main result of this section.

\begin{thm} \label{main-theorem}  Let   $G$ be  a graph with $\alpha(G)\ge 3$ which is not a join of its two proper  subgraphs. Then, $G$ is a locally triangle-free graph in $W_2$ if and only if $G$ is a triangle-free graph in $W_2$, or
$G$ is isomorphic to one of $Q_9, Q_{12}, P_{10}$, or $P_{12}$.
\end{thm}
\begin{proof} If $G$ is a triangle-free graph in $W_2$ or one of $Q_9, Q_{12}, P_{10}$, or $P_{12}$, then we can check that $G$ is also a locally triangle-free graph in $W_2$.

Conversely,  assume that $G$ is a locally triangle-free graph in $W_2$. It suffices to prove that if $G$ is not triangle-free, then $G$ is one of $Q_9, Q_{12}, P_{10}$, or $P_{12}$. We now consider two possible cases:

\medskip

{\it Case 1: For every triangle $(uvw)$ of $G$, we have $N_G(u)\setminus N_G[v]\in \Delta(G)$.}

\medskip

Let $(abc)$ be a triangle of $G$ and let
 $A:=N_G(a)\setminus N_G[b], B := N_G(b)\setminus N_G[a]$, and $C := N_G(c)\setminus N_G[a]$
so that $A,B,C\in \Delta(G)$. By Lemma $\ref{singleton}$, we have $V(G_{ab})\subseteq N_G(c)$, so $G_{ab}$ is an induced subgraph of $G[C]$. Thus,  $G_{ab}$ is a totally disconnected graph. By Lemma $\ref{Qs}$, we have $G$ is isomorphic to either $Q_9$ or $Q_{12}$.

\medskip

{\it Case 2: There is a triangle $(abc)$ of $G$ such that $N_G(a)\setminus N_G[b]\notin \Delta(G)$.}

\medskip

Let  $A:=N_G(a)\setminus N_G[b] \text{ and } B := N_G(b)\setminus N_G[a],$ so that  $A\notin \Delta(G)$. If $B\in \Delta(G)$, then $G$ is isomorphic to $Q_{10}$ by Lemma $\ref{P10}$. Otherwise, $B\notin \Delta(G)$, and so $G$ is isomorphic to $P_{12}$ by Lemma $\ref{P12}$. The proof of the theorem is complete.
\end{proof}

\section{Buchsbauness of second powers of edge ideals}

Let $R := K[x_1,\ldots, x_n]$ be the polynomial ring over a field $K$ and let $G$ be a graph with vertex set $\{x_1,\ldots,x_n\}$. In this section we characterize graphs $G$ such that $I(G)^2$ are Buchsbaum. First we characterize  locally triangle-free Grorenstein graphs.

\begin{thm} \label{thmGorenstein}  Let $G$ be a locally triangle-free graph. Then  $G$ is  Gorenstein if and only if $G$ is either a triangle-free graph in $W_2$, or $G$ is isomorphic to one of  $C_n^c$ ($n\geqslant 6$), $Q_{9}$, $Q_{12}$, $P_{10}$ or $P_{12}$.
\end{thm}
\begin{proof} We consider three cases:

\medskip
{\it Case 1: $\alpha(G)=1$.} Then $G$ is a complete graph. It is well known that all Gorenstein complete graphs are just $K_1$ and $K_2$, so the theorem holds true in this case.

\medskip
{\it Case 2: $\alpha(G)=2$.} Note that $C_n^c$ is Gorenstein for any $n\geqslant 4$ because the geometric realization of $\Delta(C_n^c)$ is isomorphic to the $1$-dimensional sphere. Thus, the theorem follows from \cite[Lemma 2.5]{HT} and Proposition \ref{small-cases}.

\medskip
{\it Case 3: $\alpha(G) \geqslant 3$.} If $G$ is Gorenstein,  by \cite[Proposition 3.3 in Chapter 0 and Corollary 4.2 in Chapter II]{S} the complex $\Delta(G)$ is connected. Together with Lemma \ref{disconnected}, it follows that $G$ is not a join of its two proper subgraphs.  By Theorem $\ref{main-theorem}$, $G$ is either a triangle-free graph in $W_2$ or isomorphic to one of $Q_{9}$, $Q_{12}$, $P_{10}$ or $P_{12}$.

Conversely, if $G$ a triangle-free graph in  $W_2$, then $G$ is Gorenstein by \cite[Theorem 3.4]{HT}. If  $G$ is isomorphic to   $Q_{9}$ (resp.   $P_{10}$, and $P_{12}$) as indicated in Figure \ref{fig_fourgraphs}, its edges are diagonals of a {\it triaugmented triangular prism}  (resp. {\it gyroelongated square bipyramid}, and {\it icosahedron}) shown in Figure \ref{polytopes}.  In other words,  all maximal independent sets of $G$ are triangles of such polytopes, and the  geometric realization of $\Delta(G)$ is isomorphic to the 2-dimensional sphere $\mathbb S^2$.  Therefore, $G$ is  a Gorenstein graph.

\begin{figure}[H]
 \begin{tabular}{ccccc}
\includegraphics[scale=0.46]{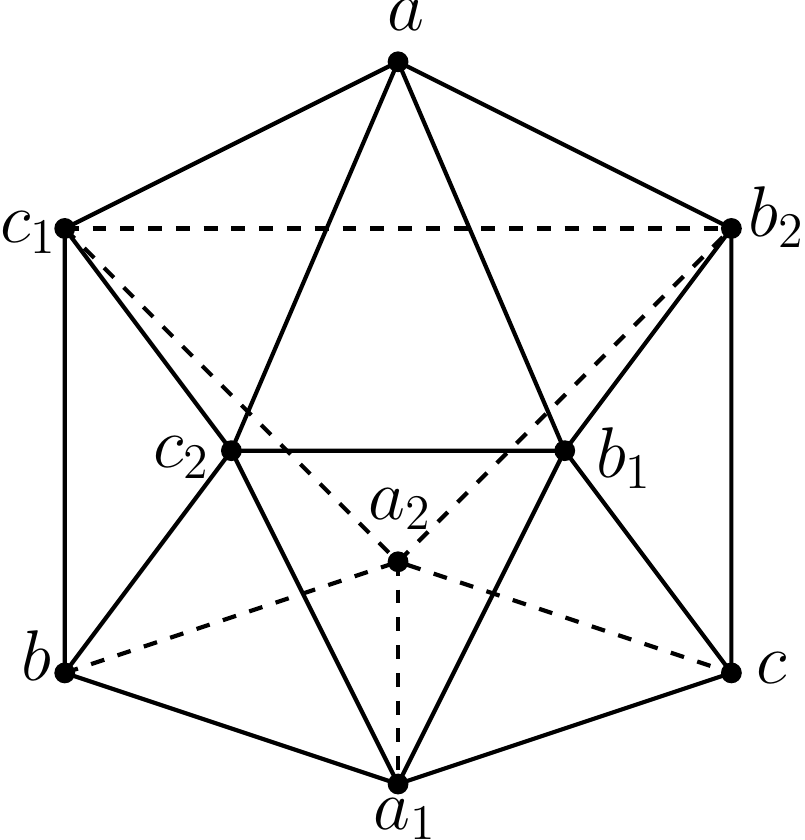}
&\qquad &
\includegraphics[scale=0.5]{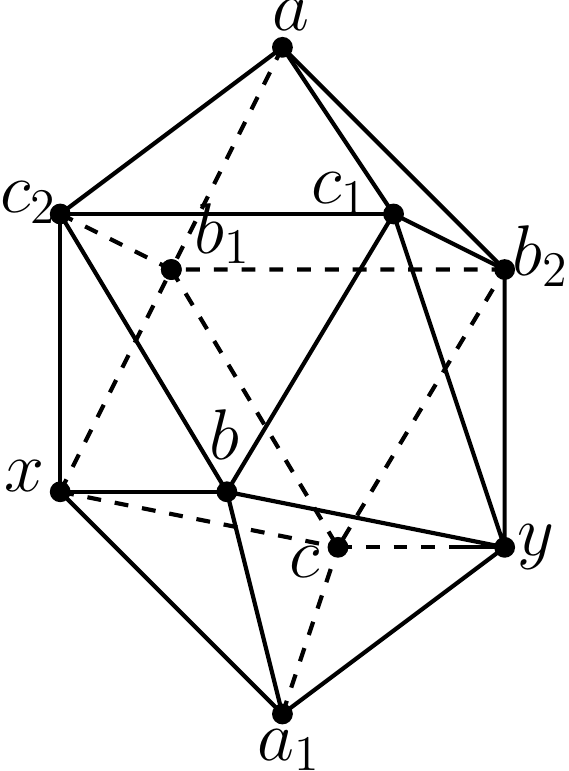}
&\qquad &
\includegraphics[scale=0.55]{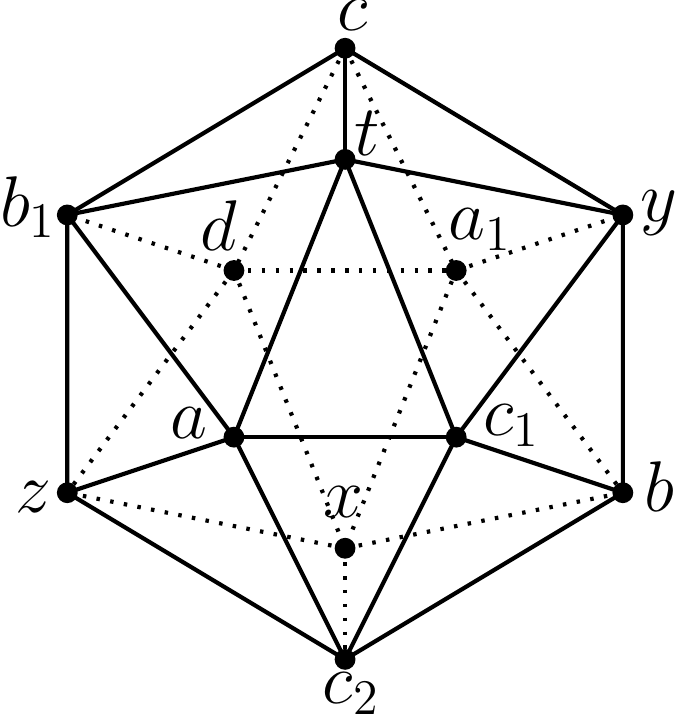}
\\[5pt]
{Triaugmented}
&&
{Gyroelongated}
&&
{Icosahedron}
\\
{triangular prism}
&&
{square bipyramid}
\end{tabular}
\caption{Three polytopes associated with graphs $Q_9,P_{10},P_{12}.$}
\label{polytopes}
\end{figure}

Finally, if  $G$ is isomorphic to $Q_{12}$ (see Figure \ref{fig_fourgraphs}), its maximal independent sets are listed as follows:
\begin{center}
\begin{tabular}{lllllll}
$ab_1b_2b_3$&
$ab_1b_2c_3$&
$ab_1b_3c_2$&
$ab_1c_2c_3$&
$ab_2b_3c_1$&
$ab_2c_1c_3$&
$ab_3c_1c_2$
\\
$ac_1c_2c_3$&
$ba_1a_2a_3$&
$ba_1a_2c_3$&
$ba_1a_3c_2$&
$ba_1c_2c_3$&
$ba_2a_3c_1$&
$ba_2c_1c_3$
\\
$ba_3c_1c_2$&
$bc_1c_2c_3$&
$ca_1a_2a_3$&
$ca_1a_2b_1$&
$ca_1a_3b_3$&
$ca_1b_1b_3$&
$ca_2a_3b_2$
\\
$ca_2b_1b_2$&
$ca_3b_2b_3$&
$cb_1b_2b_3$&
$a_1a_2b_1c_3$&
$a_1a_3b_3c_2$&
$a_1b_1b_3c_2$&
$a_1b_1c_2c_3$
\\
$a_2a_3b_2c_1$&
$a_2b_1b_2c_3$&
$a_2b_2c_1c_3$&
$a_3b_2b_3c_1$&
$a_3b_3c_1c_2$&
\end{tabular}
 \end{center}
It implies that  the geometric realization of $\Delta(G)$ is a triangulation of $3$-dimensional sphere $\mathbb S^3$ with face vector  $(12, 45, 66, 33)$ (see \cite{LT}), so $G$ is a Gorenstein graph. The proof of the theorem is complete.
\end{proof}

\begin{lem}\label{discreteW2} Assume that  $G$ is  a well-covered, locally triangle-free, connected graph such that
\begin{enumerate}
\item $\alpha(G)\geqslant 3$; and 
\item $G$ is not a join of its two  proper subgraphs; and 
\item Each nontrivial connected component of $G_v$ is in $W_2$ for every vertex $v$.
\end{enumerate}
Then, $G$ is in $W_2$.
\end{lem}
\begin{proof} Assume on the contrary that $G$ is not in $W_2$. By \cite[Lemma 2]{FHN},  there would be an  independent set $S$ of $G$ such that  $G_S$ has an isolated vertex, say $b$. Let $a$ be a vertex in $S$ so that $b$ is a vertex of $G_a$. Let $G'$ be the connected component of $G_a$ such that $b\in V(G')$. If $G'$ is nontrivial, then $G'$ is in $W_2$. Let $S' = S \cap V(G')$. Then, $S'$ is  an independent set of $G'$ and $G'_{S'}$ is an induced subgraph of $G_S$. But then, $b$ is an isolated vertex of $G'_{S'}$ which contradicts Lemma $\ref{locally-W2}$. Thus, $G'$ is a trivial graph. In other words, $b$ is an isolated vertex of $G_a$.

Let $A:=N_G(a)$ and $B:= N_G(b)$. Then $ab\notin E(G)$ and  $B\subseteq A$. Note that $A$ and $B$ are not empty since the graph $G$ is connected.  Let $H:=G_{\{a,b\}}$.  By Lemma \ref{locally-well-covered},  $H$ is well-covered with $\alpha(H) = \alpha(G)-2\geqslant 1$. In particular, $H\ne\emptyset$.

\medskip

{\it Claim:  Each  vertex of $H$ is adjacent to all  vertices in $B$.}

\medskip

Indeed, assume on the contrary that $uv\notin E(G)$ for some $u\in V(H)$ and $v\in B$. Since $G_v$ has a connected component, say $\Gamma$, which contains a 2-path $avb$, $\Gamma \in W_2$ and   $b$ is an  isolated vertex of $\Gamma_a$, which contradicts Lemma $\ref{locally-W2}$, and the claim follows.

\medskip

Let  $Z := A\setminus B$. Then, $Z = N_{G_b}(a)$, and so $Z\in\Delta(G)$ because $G_b$ is a triangle-free graph.  Now if either $Z =\emptyset$ or $B\subseteq N_G(z)$ for all $z\in Z$, then by Claim above  we would have
$$G = G[B] * G[\{a,b\}\cup Z\cup V(H)],$$
a contradiction. Hence, $Z\ne \emptyset$ and there is $z\in Z$ such that $B\not \subseteq N_G(z)$.

Next we consider the graph $G_z$.  Let   $H' := H \setminus N_G(z)$ and $B' := B\setminus N_G(z)$ (see Figure \ref{graphGz}). Then, $B'\ne \emptyset$. Let
\begin{align*}
 &Z_1 := \{z_1\in Z\setminus \{z\}\mid    z_1b'\in E(G) \text{ for some } b'\in  B'\},\\
&Z_2:= \{z_2\in Z\setminus \{z\}\mid z_2h'\in E(G) \text{ for some } h'\in  V(H')\}, \text{ and }  \\
 &Z_3:=Z\setminus (\{z\} \cup Z_1\cup Z_2).
\end{align*}

\begin{figure}[H]
\includegraphics[scale=0.6]{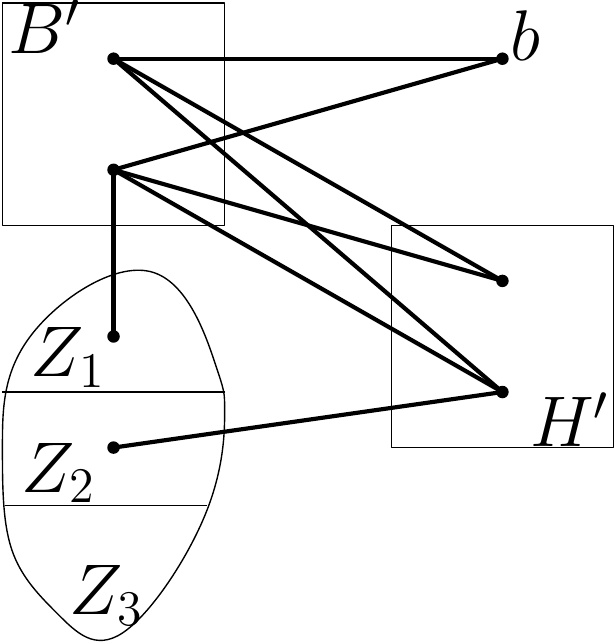}\\
\caption{The graph $G_z$.}
\label{graphGz}
\end{figure}

Note that all $Z_1, Z_2, Z_3$ are independent sets of $G$. Inside the triangle-free graph $G_z$ we have $B' = N_{G_z}(b)$ and  $B'\in \Delta(G)$. Furthermore, by Claim above we imply that   $Z_1 \cap Z_2 = \emptyset, N_G(Z_2) \cap B' = \emptyset, N_G(Z_1)\cap V(H') =\emptyset$, and $H'$ is totally disconnected.

It follows that $Z_3$ is the set of isolated vertices in $G_z$ and $G_z\setminus Z_3$ is a connected bipartite graph with bipartition $(B'\cup Z_2, V(H')\cup Z_1\cup \{b\})$. Since this bipartite graph is a nontrivial component of $G_z$, it is in $W_2$. By Lemma $\ref{bipartite}$, it is just an edge. Thus, $H'=\emptyset$, $Z_1=Z_2=\emptyset$ and $|B'| = 1$.

Finally, since $G_b$ is a connected bipartite graph with bipartition $(V(H)\cup\{a\}, Z_3\cup \{z\})$, it is an edge by Lemma $\ref{bipartite}$. It follows that $V(H) = \emptyset$, i.e. $H=\emptyset$, a contradiction. The proof of the lemma is complete.
\end{proof}

We are now in position to prove the main result of this paper.

\begin{thm}\label{thmBuchsbaum}   Let   $G$ be  a simple graph. Then,  $I(G)^2$ is  Buchsbaum if and only if
$G$ is a triangle-free graph in $W_2$, or $G$ is isomorphic to one of $K_n$  ($n\ge 3$), $C_n^c$ ($n\ge 6$), $B_n$  ($n\geqslant 4$),  $Q_9, Q_{12}, P_{10}$ or $P_{12}$.
\end{thm}
\begin{proof} If $\alpha(G)=1$, then $G$ is a complete graph, and so  $I(G)^2$ is always Buchsbaum.

If $\alpha(G)=2$, by \cite[Theorem 4.12]{MN}, $I(G)^2$ is Buchsbaum if and only if $\Delta(G)$ is an $n$-cycle, or an $(n-1)$-path ($n\geqslant 5$). Therefore,  $G$ is isomorphic to one of $B_n$  ($n\ge 4$), or $C_n^c$ ($n\ge 5$).

Assume that $\alpha(G) \geqslant 3$. By \cite[Theorem 3.12]{HMT} we have  $I(G)^2$ is Buchsbaum if and only if  $G$ is Cohen-Macaulay and $I(G_v)^2$ is Cohen-Macaulay for all $v\in V(G)$.

If $ I(G)^2$ is Buchsbaum,  then $G$ is Cohen-Macaulay, and then $G$ is well-covered. Recall that $I(G_v)^2$ is Cohen-Macaulay for every vertex $v$. Since $I(G_v)^2$ is Cohen-Macaulay if and only if  every nontrivial connected  component of $G_v$ is triangle-free  in $W_2$ due to \cite[Theorem $4.4$]{HT}, we have $G$ is a well-covered locally triangle-free graph. Since $G$ is Cohen-Macaulay, $\Delta(G)$ is connected by \cite[Proposition 3.3 in Chapter 0 and Corollary 4.2 in Chapter II]{S}. Thus, $G$ is not a join of its two proper subgraphs by Lemma $\ref{disconnected}$. By Lemma \ref{discreteW2}, $G$ is in $W_2$. Together with Theorem $\ref{main-theorem}$, we have  $G$ is either a triangle-free graph or isomorphic to one of $Q_9, Q_{12}, P_{10}$ or $P_{12}$.

Conversely, assume first that $G$ is a triangle-free graph in $W_2$. Then, $I(G)^2$ is Cohen-Macaulay (and so is Buchsbaum) by \cite[Theorem $4.4$]{HT}, and the theorem follows.

If $G$  is isomorphic to one of $Q_9, Q_{12}, P_{10}$ and $P_{12}$, then  $G$ is a locally triangle-free Gorenstein graph by Theorem \ref{thmGorenstein}. In particular, $G$ is a Cohen-Macaulay graph and in $W_2$.  Thus, for each vertex $v$ we have $G_v$ is a triangle-free graph in $W_2$, so $I(G_v)^2$ is Cohen-Macaulay by \cite[Theorem $4.4$]{HT}. Thus, $I(G)^2$ is Buchsbaum, and the proof of the theorem is complete.
\end{proof}

\subsection*{Acknowledgment} The authors are    supported by the NAFOSTED (Vietnam) under grant number 101.04-2015.02. The second author is also partially supported by VAST.DLT 01/16-17.  We would like to thank the Korea Institute for Advanced Study  for financial support  and  hospitality  during our visit  in 2016, when we started to work on this paper.

\end{document}